\documentclass[11pt]{article}

\usepackage{amsrefs}
\usepackage{amsmath,amssymb,amsthm}
\usepackage{enumerate,pxfonts,tikz}
\usepackage[all,cmtip]{xy}
\usepackage[applemac]{inputenc}
\usepackage[colorlinks=true, linkcolor=black, citecolor=black]{hyperref}
\usepackage{marvosym}

\usepackage[OT1]{fontenc}
\DeclareFontFamily{OT1}{pzc}{}
\DeclareFontShape{OT1}{pzc}{m}{it}{<-> s * [1.35] pzcmi7t}{}
\DeclareMathAlphabet{\mathcal}{OT1}{pzc}{m}{it}

\newcommand \Ad{\operatorname{Ad}}
\newcommand \al{\alpha}
\newcommand \adm{\mathrm{adm}}

\newcommand \aft{\circ}
\newcommand \be{\beta}
\newcommand \bs{\backslash}
\newcommand \C{{\mathbb C}}

\newcommand \CA{\mathcal{A}}

\newcommand \CD{\mathcal{D}}

\newcommand \CF{\mathcal{F}}

\newcommand \Cg{\mathcal{g}}

\newcommand \CM{\mathcal{M}}

\newcommand \Cn{\mathcal{n}}

\newcommand \CP{\mathcal{P}}

\newcommand \CR{\mathcal{R}}
\newcommand \CS{\mathcal{S}}

\newcommand \Cz{\mathcal{z}}
\newcommand \cusp{\mathrm{cusp}}

\newcommand \ds{\displaystyle}
\newcommand \Eig{\operatorname{Eig}}
\newcommand \Eis{\operatorname{Eis}}
\newcommand \End{\operatorname{End}}

\newcommand \fg{\mathrm{fg}}
\newcommand \Ga{\Gamma}
\newcommand \GL{\operatorname{GL}}
\newcommand \ga{\gamma}

\newcommand \Hom{\operatorname{Hom}}
\newcommand \Id{{\rm Id}}
\newcommand \la{\lambda}
\newcommand \La{\Lambda}
\newcommand \Lie{\operatorname{Lie}}

\newcommand \mqed{\tag*\qedhere}
\newcommand \N{{\mathbb N}}
\newcommand \ol{\overline}
\newcommand \om{\omega}
\newcommand \Om{\Omega}

\newcommand \pos{\mathrm{pos}}

\newcommand \R{\mathbb{R}}
\newcommand \SL{\operatorname{SL}}
\newcommand \sm{\smallsetminus}

\newcommand \strong{\mathrm{strong}}
\newcommand \supp{\operatorname{supp}}

\newcommand \ul{\underline}
\newcommand \vol{\operatorname{vol}}
\newcommand \weak{\mathrm{w-cusp}}
\newcommand \what{\widehat}
\newcommand \Z{\mathbb{Z}}

\renewcommand \1{{\bf1}}

\renewcommand \Pr{\operatorname{Pr}}

\renewcommand \({\big(}
\renewcommand \){\big)}
\renewcommand \[{\left(}
\renewcommand \]{\right)}

\newcommand{\e}
[1]{\emph{#1}\index{#1}}

\newcommand{\norm}
[1]{\left\|#1\right\|}

\renewcommand{\sp}
[1]{\left\langle #1\right\rangle}

\newcommand{\tto}
[1]{\stackrel{#1}{\longrightarrow}}

\newtheorem{theorem}{Theorem}[section]

\newtheorem{lemma}[theorem]{Lemma}
\newtheorem{corollary}[theorem]{Corollary}
\newtheorem{proposition}[theorem]{Proposition}

\theoremstyle{definition}
\newtheorem{definition}[theorem]{Definition}

\newtheorem{remark}[theorem]{Remark}

\begin{document}

\pagestyle{myheadings} \markright{SPECTRAL THEOREM}

\title{A spectral theorem for compact representations and non-unitary cusp forms}
\author{Anton Deitmar \footnote{This paper has been written during a stay at the Institute for Fundamental Mathematics at the NTT Research \& Development Center, Tokyo, Japan. The author thanks the Institute and its members for the warm hospitality he has received there.}}
\date{}
\maketitle

{\bf Abstract:}
We show that a compact representation of either a semisimple Lie group or a totally disconnected group has a filtration with irreducible subquotients of finite multiplicity.
In the Lie group case we show the stronger assertion, that it has 
an orthogonal decomposition into summands of finite lengths.
This generalises and  simplifies  a number of more special spectral theorems in \cites{Mueller, DeiMon, Venkov}.
We apply it to the case of cusp forms, thus settling the spectral theory for the space of non-unitary twisted cusp forms.

$ $


$ $


$ $ 

{\bf Data availability statement:} The author confirms that the data supporting the findings of this study are available within the article.

\newpage

\tableofcontents

\section*{Introduction}

A Hilbert representation $\pi$ of a locally compact group $G$ is called \e{compact}, if for every test function $f\in C_c^\infty(G)$ the operator 
$\pi(f)$ is compact.
If $\pi$ is compact and unitary, it follows that $\pi$ decomposes discretely, i.e., it is a direct sum of irreducible representations.
If $\pi$ is  not unitary, this form of discreteness is not to be expected in general.
One might however put forward the question, whether a compact $\pi$ must have a filtration with irreducible quotients as in \cite{Dei}.
In the present paper, this question is settled  for semi-simple Lie groups  and for totally disconnected groups. 
The result is then applied to the space of twisted  {cusp forms} $L^2_\cusp(\Ga\bs G,\om)$, which is a subspace of the section space of a certain vector bundle. More precisely, a given finite-dimensional representation $\om:\Ga\to\GL(V)$ gives rise to a flat vector bundle $E_\om$ over $\Ga\bs G$.
The sections of $E_\om$ can be considered as functions $\phi:G\to V$ with  $\phi(\ga x)=\om(\ga)\phi(x)$ for all $\ga\in\Ga$, $x\in G$.
If the  connection on $E_\om$ is a metric connection, or, equivalently, the representation $\om$ is unitary, then there is a natural definition of a Hilbert space $L^2(E_\om)$ of sections, acted upon by $G$, yielding a  unitary representation of $G$.
The spectral theory of this representation is, if $\Ga\bs G$ is compact, very similar to the untwisted case. 
In general, there are four possible situations

\begin{center}
\begin{tabular}{c|c|c}
&$\Ga\bs G$ compact & $\Ga/G$ non-compact\\
\hline
$\om$ unitary&(I)&(II)\\
\hline
$\om$ non-unitary&(III)&(IV)
\end{tabular}
\end{center}
The case of concern to this paper is (IV).
The case (I) differs  not much from the untwisted case.
For (II), a complete spectral analysis of the attached Laplacian has, for the group $\SL_2(\R)$, been given by A. Venkov \cite{Venkov}.
An investigation of the case (III) has been initiated by W. Müller \cite{Mueller}.
In the case (IV), not much is known.
It is not even clear how to define the $L^2$-space, as there is no preferred metric class and no $G$-invariant metric in general.
Only in the case when the $\Ga$-representation $\om$ extends to the group $G$, one has a canonical approach by transforming the sections $\phi$ to $\Ga$-invariant functions and use  classical results, see Section \ref{sect1}.
If the rank of $G$ is $\ge 2$ and $\Ga$ is an irreducible lattice, this is automatically the case by super-rigidity.

In the untwisted situation, the space $L^2(\Ga\bs G)$ is a direct sum of the space of cusp forms and a space spanned by Eisenstein series. 
For a complete spectral theory, one first extends the Eisenstein series to all of $\C$ and then identifies the Eisenstein space with residues of and integrals over Eisenstein series.
In the twisted case, the analytic continuation of Eisenstein series has been provided in some cases, the rest of the programme is open.
In \cite{Cameron}, Eisenstein series are used to construct families of metrics with favourable properties.
Explicit Fourier expansions  of  eigenvectors of the hyperbolic Laplacian have been computed  in \cite{FPR}.
Analytic extension of Eisenstein series have been given in \cites{DeiMonEis,FPEisenstein}. Finally, relations to mock modular forms are in \cite{MertensRaum}.
All of these are, however, only first steps of installing a  spectral theory for the twisted, non-compact case.
In the present paper we turn our attention to the space of cusp forms.
As a first problem, it is not even clear how to define cusp forms for non-unitary twists.

In this paper, we give two different definitions of cusp forms, depending on the behaviour of the twisting data $\om$ at the cusps. 
We call it the strong and weak cusp forms, since, in situations where both are defined, one condition is stronger than the other.
For both notions, we show that the respective representations are compact and then use the general spectral theorem mentioned above, to conclude that  
the space of cusp forms is an  orthogonal  direct sum of representations of finite length.
In the special case when the twisting representation $\om$ has no fixed points at the cusps, indeed every automorphic form is weakly cuspidal and thus the whole $L^2$-space decomposes discretely.

\section{The Spectral Theorem for Lie groups}\label{Sec1}

\begin{definition}
By a \e{Hilbert representation} of $G$ we mean a continuous representation on a Hilbert space.
In this paper, we consider Hilbert representations instead of Banach space representations, as it is sometimes required and fits to the applications.
\end{definition}

\begin{definition}
A representation of a locally compact group $G$ is called a \e{compact representation}, if the convolution algebra $C_c^\infty(G)$ acts by compact operators.
The space of test functions $C_c^\infty(G)$ has been defined by Bruhat in \cite{Bruhat}, see also \cite{Tao}.
For a Lie group, it is the usual space of test functions, for a totally disconnected group, it is the space of all locally constant, compactly supported functions.
Examples of compact representations are the right translation representation on $L^2(\Ga\bs G)$, where $\Ga$ is a co-compact lattice.
If $G$ is discrete, however, the only compact representations are the finite-dimensional representations.
\end{definition}

\begin{definition}
For the rest of the section, let $G$ be a semi-simple Lie group with finite center and finitely many components.
Let $K$ be a maximal compact subgroup.
Let $R$ be a representation on a Hilbert space $V$.
If the restriction to $K$ is unitary, we call the representation \e{$K$-unitary}. The latter can be achieved by changing the Hilbert structure to an equivalent one which is obtained by Haar integration over $K$. 
\end{definition}

\begin{definition}
A vector $v\in V$ for a representation $(\pi,V)$ is called \e{$K$-finite}, if the orbit $\pi(K)v$ spans a finite-dimensional space.
The space of $K$-finite vectors is denoted by $V_{K}$.
A vector $v$ is called \e{smooth}, if the map $x\mapsto\pi(x)v$ is infinitely differentiable.
The space of smooth vectors is denoted by $V^\infty$.
We finally write $V_K^\infty=V_K\cap V^\infty$.
\end{definition}

\begin{definition}\label{DefKendlich}
The group $K\times K$ acts on $C_c^\infty(G)$ by right and left translations respectively.
We say that $f\in C_c^\infty(G)$ is \e{$K$-finite}, if the span of the $K\times K$ orbit of $f$ is finite dimensional.
We write $C_{c,K}^\infty$ for the convolution algebra of $K$-finite functions in $C_c^\infty(G)$.
\end{definition}

\begin{definition}
Let $(\pi,V)$ be a representation of $G$.
The space $V_K^\infty$ is not $G$-stable in general, but it defines a $(\Cg,K)$-module 
\cite{Wallach}.
We call two representations \e{infinitesimally equivalent}, if their $(\Cg,K)$-modules are isomorphic.
\end{definition}

\begin{remark}
Note that, since $\eta$ is chosen right $K$-invariant,
the representation on $L^2(\Ga\bs G,\om)$ is $K$-unitary.
\end{remark}

\begin{definition}
A  $K$-unitary Hilbert representation of $G$ is called \e{admissible}, if every $K$-isotypical component is finite-dimensional.
Let  $\what G_\adm$ denote the set of all irreducible admissible representations of $G$ up to infinitesimal equivalence \cite{Wallach}.
Finally, let $\what G_\fg$ denote the set of those infinitesimal equivalence classes, whose $(\Cg,K)$-modules are finitely generated.
Each of such has a finite filtration with irreducible quotients  \cites{Knapp,Wallach}. 
\end{definition}

\begin{definition}\label{def2.8}
Let $\Cz$ denote the center of the universal enveloping algebra $U(\Cg)$. Then $\Cz\cong [D_1,\dots,D_R]$, where $R=\mathrm{rank}(G)$ and we can assume that $D_1=C$ the Casimir element.
Hence the set $\Hom_{\mathrm{alg}}(\Cz,\C)$ of algebra homomorphisms $\la:\Cz\to \C$ can be identified with $\C^R$.

For an operator $T$ on a linear space $V$ and $\la\in\C$ we write
$$
\Eig^\infty(T,\la)=\bigcup_{n=1}^\infty\ker(T-\la)^n
$$
and we call this space the \e{generalised eigenspace} for $\la\in\C$.
For a representation $(R,H)$ of $G$ and $\la\in\C^R$ let $H_\la$ denote the closure of the space
\begin{align*}
\Eig^\infty(\la)=\bigcap_{j=1}^R\Eig^\infty(C_j,\la_j)
\end{align*}\end{definition}

 \begin{theorem}\label{nextprop}.
 Let $G$ be a semisimple Lie group with finite center and finitely many components.
Let $(R,H)$ be a $K$-unitary compact Hilbert representation of $G$.
Then the representation $R$ is a direct sum of the subrepresentations $H_\la$, $\la\in\C^R$, i.e., the space
$
\ds\bigoplus_{\la\in\C^R}H_\la
$
is dense in $H$.
Only countably many of the $H_\la$ are non-zero and each $H_\la$ has a finite filtration with irreducible quotients.
\end{theorem}

\begin{proof}
As $K$ acts unitarily, we get an orthogonal decomposition into $K$-isotypes
$$
H=\bigoplus_{\tau\in\what K}H(\tau).
$$
A given $\tau\in\what K$ defines a vector bundle $E_\tau=\(G\times V_\tau\)/K$. 
The sections can be viewed as functions $s:G\to V_\tau$ with $s(xk)=\tau(k^{-1})s(x).$
The Casimir operator $C$ acts on those.
Now $C=C_\pos+C_K$, where $C_K$ is the Casimir of $K$ and $C_\pos=\sum_{j=1}^nX_j^2$, where $(X_j)$ is an orthonormal basis of Killing orthocomplement of $\Lie(K)$ in $\Lie(G)$.
The operator $C_K$ acts by a scalar on the sections of $E_\tau$, whereas $C_\pos$ acts as $\Delta\ \Id$, where $\Delta$ is the Laplace operator of $X$ and $\Id$ is the identity section of the endomorphism bundle of $E_\tau$.
For distinction, we write $C_\tau$ for the differential operator induced by $C_\pos$ on $E_\tau$.
Let $f$ be an entire function such that $f(z^2)$ is a Payley-Wiener function.
Then $f(C_\tau)$ is an integral operator with smooth kernel and finite propagation speed \cite{CGT}. It also is $G$-invariant, hence given by right-convolution with some function in $C_c^\infty(G)(\tau,\tau)$ of $K$-type $\tau$ on both sides.
By our condition, the operator $f(C_\tau)$ is compact.
Hence every spectral value $0\ne\mu\in\sigma\(f(C_\tau)\)$ is an eigenvalue and 
the generalised eigenspace
$
\Eig^\infty\(f(C_\tau),\mu\)$
is finite-dimensional.
We have $\Eig^\infty(C_\tau,\mu)\subset\Eig^\infty\(f(C_\tau),f(\mu)\)$, and hence $\Eig^\infty(C_\tau,\la)$ is finite-dimensional for every $\la\in\C$.
Since $f\(\sigma(\Delta)\)=\sigma\(f(\Delta)\)$, these eigenvalues make up the spectrum of $\Delta$.

It follows that $H$ is the closure of the direct sum of finite dimensional spaces
$$
\bigoplus_{\substack{\tau\in\what K\\ \la_1\in\C}}\Eig^\infty(C|_{H(\tau)},\la_1).
$$
This space is acted upon by the algebra $\Cz$, respecting the decomposition. 
Hence the space equals
$$
\bigoplus_{\substack{\tau\in\what K\\ \la\in\C^R}}\Eig^\infty(\la)(\tau).
$$
We get
$$
H_\la=\ol{\bigoplus_{\tau\in\what K}\Eig^\infty(\la)(\tau)}
$$
Note that these spaces remain pairwise linearly independent, since by the orthogonality of $K$-types, the space $H_\la$ can be identified with a subspace of $\prod_{\tau\in\what K}\Eig^\infty(C_\tau,\la)\subset \prod_{\tau\in\what K}H(\tau)$, in which space they still are linearly independent.
The space $H_\la$ is closed, $G$-stable and admissible.
We claim that it has a finite filtration with irreducible quotients.
For this take some $\tau\in\what K$ with $H_\la(\tau)\ne 0$.
Let $W\subset H_\la(\tau)$ of minimal positive dimension such that there exists a subrepresentation $S\subset H_\la$ with $S(\tau)=W$.
Let $T\subset S$ be a maximal subrepresentation with $T(\tau)=0$.
Then $S/T$ is an irreducible subquotient of $H_\la$.
Repeat with $H_\la/S$ until you get a finite filtration
$$
0=F_0\subset F_1\subset\dots\subset F_{2k+1}
$$
such that $F_{2j}/F_{2j-1}$ is irreducible and $F_{2j-1}/F_{2j-1}$ does not have $K$-type $\tau$.
For these subquotients you repeat with another $\ga\in\what K$.
For given $\la$ there are only finitely many classes in $\what G_\adm$ having $\la$ as generalised eigenvalue.
Hence the process stops after trying finitely many classes $\tau$.
\end{proof}

\section{Filtrations}

The Spectral Theorem of Section \ref{Sec1} is the best one can hope for in the case of a non-unitary representation.
In general, one cannot expect a decomposition, but one might obtain a filtration with irreducible subquotients.

\begin{definition}
Two elements $a<b$ of an ordered set are called \e{neighbours}, if $[a,b]=\{a,b\}$, i.e., there are no elements between them.
For a linearly ordered set $L$ let $L_n$ denote the subset of all elements possessing a neighbour.
A \e{ladder} is a linearly ordered set $L$ such for every $a\in L\sm L_n$ one has
\begin{align*}
a&=\sup\big\{b\in L_n: b<a\big\}=\inf\big\{c\in L_n:a<c\big\}.
\end{align*}
Examples are $\N$, $\Z$, sections of the class of ordinal numbers or non-standard models of number theory. 
\end{definition}

\begin{definition}
Let $(R,V)$ be a Hilbert representation of a topological group $G$.
A \e{filtration} of $R$ is a family of subrepresentations  $(F_a)_{a\in L}$ for a partially ordered set $L$, such that $F_a\subset F_b$ if $a<b$.
It is called a \e{ladder filtration}, if $L$ is a ladder and $F_a\ne F_b$ for all $a<b$.
Finally, a ladder filtration is called a  \e{complete filtration}, if
\begin{enumerate}[\rm(a)]
\item If $a<b$ are neighbours, then $F_b/F_a$ is irreducible.
\item If $L=A\sqcup B$ is a decomposition with $A<B$, then 
$$
\ol{\bigcup_{a\in A}F_a}=\bigcap_{b\in B}F_b.
$$
\end{enumerate}
\end{definition}

\begin{definition}
Let a complete filtration $\CF=(F_a)_{a\in L}$ be given.
For an irreducible $\pi$ the \e{$\CF$-multiplicity} $N_\CF(\pi)$ is defined to be the supremum of the number of all neighboured pairs $a<b$ such that $\pi\cong F_b/F_a$.
\end{definition}

\begin{definition}
For an irreducible $\pi$ define the \e{multiplicity} $N(\pi)\in\N_0\sqcup\{\infty\}$ of $\pi$ in $V$ to be the supremum of the $k\in \N_0$ such that
\begin{align*}
\exists\ {F_1\subset\cdots\subset F_{2k}\text{ subrepresentations with }} F_{2j}/F_{2j-1}\cong\pi
\tag*{$(*)$}
\end{align*}for $1\le j\le k$.
We then have $N_\CF(\pi)\le N(\pi)$ for every irreducible $\pi$ and every complete filtration $\CF$.
A finite filtration as in $(*)$ is called a $\pi$-filtration.
\end{definition}

\begin{lemma}
Let $(R,V)$ be a Hilbert representation.
Suppose that every subquotient of $R$ possesses an irreducible subquotient.
Then $R$ admits a complete filtration and we have $N(\pi)=N_\CF(\pi)$ for every complete filtration $\CF$.
\end{lemma}

\begin{proof}
We order the set $\CS$  of all ladder filtrations by setting $(F_a)_{a\in L}\le (G_b)_{b\in M}$ if $L$ is a subset of $M$ and $G_a=F_a$ for every $a\in L$.
The Lemma of Zorn yields a maximal element $\CF$, the completeness of which is easily checked.

The second assertion follows if we show that every $\pi$-filtration $F_1\subset\cdots\subset F_{2k}$ with length $k< N(\pi)$ can be extended to a $\pi$-filtration of length $k+1$.
So let $E_1\subset\dots\subset E_{2k+2}$ be a $\pi$ filtration of length $k+1$.
By the Schreier Refinement Theorem \cite{Baumslag}, these two filtration have a common refinement, under whose consecutive subquotients the $k+1$ copies of $\pi$ must be found. 
\end{proof}

\section{Totally disconnected groups}

In this section, let $G$ denote a totally disconnected locally compact group.
Then $G$ has a unit neighbourhood basis consisting of compact open subgroups.
The space $C_c^\infty(G)$ of test functions is the space of all compactly supported, locally constant functions on $G$.
For a compact open subgroup $K$, the indicator function $\1_K$ is in $C_c^\infty(G)$, hence for a compact representation $(R,V)$ the space $V^K$ of $K$-fixed vectors is finite-dimensional.

\begin{theorem}
Any compact Hilbert representation $(R,V)$ of $G$ has an irreducible subquotient.
It follows that $R$ possesses a complete filtration with finite multiplicities.
\end{theorem}

\begin{proof}
Let $K$ be a compact open subgroup with $V^K\ne 0$.
We can replace $V$ with the closure $\sp{V^K}$ of the $G$-span of $V^K$ and so assume that the representation is generated by $V^K$.

Let $W\subset V^K$ be maximal with the property that there exists a subrepresentation $U\subset V$ with $U^K=W$.
Let $S$ be the set of all such subrepresentation spaces $U$.
It is ordered by inclusion.
Let $L\subset S$ be a linearly ordered subset and let
$$
B=\ol{\bigcup_{U\in L}U}.
$$
The operator $P=\frac1{\vol(K)}R(\1_K)$ is a continuous projection onto $V^K$.
By continuity of $P$ we get
\begin{align*}
B^K&=P\[\ol{\bigcup_{U\in L}U}\]=\ol{P\[{\bigcup_{U\in L}U}\]}=\ol{{\bigcup_{U\in L}P(U)}}=\ol{{\bigcup_{U\in L}W}}=W
\end{align*}
So $B$ lies in $S$ and is an upper bound to $L$.
By Zorn's lemma we infer that $S$ has a maximal element $U$.
We claim that $H=V/U$ is irreducible.
For this let $E\subset H$ be a subrepresentation.
If $H^K=0$, by maximality of $U$ we infer that $H=0$.
Otherwise, we have $H^K=V^K$ and hence $H=V$.

For the assertion on finite multiplicities let $K_0$ be a compact open subgroup.
By integration, the Hilbert structure on $V$ can be changed to an equivalent, $K_0$ invariant.
This means that $K_0$, and hence all its subgroups, act unitarily on $V$.
Let $K$ be an open subgroup of $K_0$.
and Let $E/F$ be a subquotient with $(E/F)^K\ne 0$.
Then $E=F\oplus F'$ as a $K$-invariant orthogonal sum.
Hence $(E/F)^K\cong (F')^K$ and therefore the number of subquotients with $K$-invariants is finite.
Letting run $K$ through a neighbourhood basis of unity, the claim follows.
\end{proof}

\section{Flat bundles and automorphic forms}\label{sect1}
\begin{definition}
Let $G$ be a connected semisimple Lie group with finite center.
Let $\Ga\subset G$ be a lattice, i.e., a discrete subgroup of finite covolume.
For simplicity, we will assume that $\Ga$ is torsion-free, a condition that can be achieved by switching to a finite index subgroup.
There is an equivalence of categories between the finite-dimensional complex representations $\om:\Ga\to\GL(V_\om)$ and flat vector bundles $E$ over $\Ga\bs G$, which are trivial over $G$.
A representation $\om$ is mapped to
the bundle $E=E_\om$ given by
\begin{align*}
\Ga\bs(G\times V_\om)&\to \Ga\bs G,\\
\Ga(x,v)&\mapsto \Ga x.
\end{align*}
The sections then are maps $s:\Ga\bs G\to \Ga\bs (G\times V_\om)$, $s(\ga x)=[x,\phi(x)]$, where the map $\phi$ is  \e{$\om$-automorphic}, i.e.,
$$
\phi(\ga x)=\om(\ga)\phi(x)
$$ 
 for all $\ga\in\Ga$, $x\in G$.
\end{definition}

\begin{definition}\label{def1.2}
Fix a maximal compact subgroup $K$, then $X=G/K$ is the attached  symmetric space.
We shall write $d(.,.)$ for the distance function on $X$.
\end{definition}

\begin{definition}
Installing a smooth $L^2$-structure on the space of sections of $E=E_\om$ is equivalent to giving a map $x\mapsto\sp{.,.}_x$ from $G$ to the set of all inner products on the space $V_\om$ such that
$$
\sp{\om(\ga)v,\om(\ga)w}_{\ga x}=\sp{v,w}_x.
$$
Recall that two metrics $\sp{.,.}$ and $\sp{.,.}'$ are called \e{equivalent} if there exists a constant $C>0$ such that
$
\frac1C\sp{v,v}\le\sp{v,v}'\le C\sp{v,v}
$
holds for all $v\in E_\om$.
We give a canonical construction yielding a well-defined equivalence class.
\end{definition}

\begin{lemma}\label{lemma1.5}
There exists a smooth function $u:G\to [0,1]$ with the following properties:
\begin{enumerate}[\rm(a)]
\item $
\#\big\{\ga\in\Ga:\supp u\cap \supp(u\aft\ga)\ne \emptyset\big\}<\infty
$.
\item $\ds
\sum_{\ga\in\Ga}u(\ga x)=1
$ for every $x\in G$, where the sum is locally finite.
\item $u(xk)=u(x)$ for $x\in G$ and $k\in K$.
\end{enumerate}
\end{lemma}

\begin{proof}
Let $\CD\subset X$ be a Dirichlet fundamental domain for $\Ga\bs X$, i.e.,
$$
\CD=\big\{ x\in X:d(x_0,x)<d(x_0,\ga x)\ \forall_{\ga\in\Ga\sm\{1\}}\big\}
$$
for some point $x_0\in X$ called a \e{center} of $\CD$.
Then then boundary of $\CD$ consists of finitely many segments of geodesic hypersurfaces, hence has measure zero, and there exists a measurable set $R$ of representatives of $\Ga\bs X$ such that $\CD\subset R\subset\ol{\CD}$.
Choose a smooth function $\al:X\to [0,\infty)$ such that
$\al(z)\ge 1$ for $z\in\CD$ and the support of $\al$ lies in the union of $\CD$ and its finitely many neighbours.
Consider $\al$ as a $K$-invariant function on $G$.
Then the sum $\phi(x)=\sum_{\ga\in\Ga}\al(\ga x)$ is locally finite and defines a $\Ga$-invariant function $\ge 1$ everywhere.
The function $u(x)=\frac{\al(x)}{\phi(x)}$ has the desired properties.
\end{proof}

\begin{definition}
Let $u$ be as in Lemma \ref{lemma1.5} and choose an inner product $[.,.]$ on $V_\om$.
For $x\in G$ set
$$
\sp{v,w}_{u,x}=\sum_{\ga\in\Ga}u(\ga x)\ [\om(\ga)v,\om(\ga)w].
$$
\end{definition}

\begin{lemma}
The map $\sp{.,.}_{u,x}$ defines a smooth hermitean metric on $E_\om$.
If $u'$ is a second map as in Lemma \ref{lemma1.5}, then the two ensuing metrics are equivalent.
\end{lemma}

\begin{proof}
There are $\ga_1,\dots,\ga_n\in\Ga$ and $c_1,\dots,c_n>0$ such that
$$
u'(x)\le\sum_{j=1}^nc_ju(\ga_j x)
$$
holds for every $x\in G$.
Then
\begin{align*}
\norm{v}_{u',x}^2
&=\sum_{\ga\in\Ga}u'(\ga x)\ \norm{\om(\ga)v}^2\\
&\le \sum_{j=1}^nc_j\sum_{\ga\in\Ga}u(\ga_j\ga x)\ \norm{\om(\ga)v}^2\\
&\le \sum_{j=1}^nc_j\sum_{\ga\in\Ga}u(\ga x)\ \norm{\om(\ga_j)^{-1}\om(\ga)v}^2\\
\end{align*}
The norm $\norm v_j=\norm{\om(\ga_j)^{-1}v}$ on the finite dimensional space $V_\om$ is less than a constant times the original norm. The claim follows. 
\end{proof}

\begin{definition}
We define the space $L^2(\Ga\bs G,\om)$ as the Hilbert space of all measurable sections $\phi$, for which
$$
\int_{\Ga\bs G}\norm{\phi(x)}_x^2\ dx\ <\ \infty
$$
modulo null functions.
\end{definition}

\begin{lemma}
The space $L^2(\Ga\bs G,\om)$ equals the space of all measurable $\phi:G\to V_\om$ with $\phi(\ga x)=\om(\ga)\phi(x)$ and $\int_Gu(x)\norm{\phi(x)}^2\ dx<\infty$ modulo null functions, where the norm under the integral is the norm on $V_\om$.
\end{lemma}

\begin{proof}
Every section of $E_\om$ is given by a map $\phi:G\to V_\om$ with the property $\phi(\ga x)=\om(\ga)\phi(x)$ for all $\ga\in\Ga,\ x\in G$.
We compute
\begin{align*}
\int_{\Ga\bs G}\norm{\phi(x)}_x^2\,dx
&=\int_{\Ga\bs G}\sum_{\ga\in\Ga}u(\ga x)\ \norm{\phi(\ga x)}^2\,dx\\
&=\int_Gu(x)\ \norm{\phi(x)}^2\ dx.\mqed
\end{align*}
\end{proof}

\subsection*{The case of $G$-invariance}
\begin{remark}\label{rem1.9}
Here we consider the special case when the connection can be chosen $G$-invariant. 
This means that the representation $\om$ extends to $G$.
If the split rank of $G$ is $>1$ and $\Ga$ is arithmetic irreducible, this is automatic by super-rigidity \cite{Margulis,Rag}.
In this case, there is a bijection $\eta$ from the set $\CM(\Ga\bs G,\om)$ of measurable, $\om$-automorphic functions to the set $\CM(\Ga\bs G,V_\om)\cong\CM(\Ga\bs G)\otimes V_\om$ of measurable functions $\Ga\bs G\to V_\om$ given by
$$
\eta(\phi)(x)=\om(x^{-1})\phi(x).
$$
This map is $G$-equivariant in the sense that it intertwines the right translation representation $R$ with $R\otimes\om$, i.e., one has  
$$
\eta R\eta^{-1}=R\otimes\om.
$$
Now $\phi$ is in $L^2(\Ga\bs G,\om)$ if 
$$
\int_Gu(x)\norm{\phi(x)}_{V_\om}^2\ dx\ <\ \infty.
$$
On the other hand, $\eta(\phi)$ lies in $L^2(\Ga\bs G,V_\om)$, if
\begin{align*}
\int_Gu(x)\norm{\om(x^{-1})\phi(x)}^2_{V_\om}\,dx<\infty.
\end{align*}
This implies that $\eta$ does not respect $L^2$-spaces and so the respective theories of automorphic forms are unrelated, i.e., the theory of $\om$-automorphic forms is independent of the classical case, even if the  rank of $G$ is $>1$.
\end{remark}

\subsection*{Tameness}

\begin{definition}
A minimal parabolic $\CP=MAN$  is called \e{cuspidal}, if 
$\Ga_N=\Ga\cap N$ is cocompact in $N$.
Note that such a parabolic exists, as we are assuming that the lattice $\Ga$ is not cocompact.
Further, there are only finitely many $\Ga$-conjugacy classes $\CP_1,\dots,\CP_r$ of cuspidal parabolics, \cites{Borel,Borel-Ji}.
(Note that in the given sources the statement is only shown for arithmetic groups. In higher rank every lattice is arithmetic and in rank one the claim is easy.)
We fix representatives $\CP_j$ of the $\Ga$-conjugacy classes of cuspidal parabolics and call them (or their classes) the \e{cusps} of $\Ga$.
\end{definition}

\begin{definition}
Let $P=MAN$ be a cuspidal parabolic. By unipotent super rigidity, the representation $\om$ extends from $\Ga_N=\Ga\cap N$ to $N$.
We fix such an extension for every $P$ and choose it invariant under $\Ga$-conjugacy.

The representation $(\om,V_\om)$ of $\Ga$ is called \e{tame}, if for every parabolic $P$ with Langlands decomposition $P=MAN$ all eigenvalues of $\om(n)$, $n\in N$, have absolute value 1.
In \cite{FPZeta}, this property is called \e{non-expanding cusp monodromy}.
\end{definition}

\begin{proposition}
If $\om$ extends to $G$, then $\om$ is tame.
\end{proposition}

\begin{proof}
The representation $\om $ induces a representation of the Lie algebra $\Cg$. All such representations map the nilpotent algebra $\Cn=\Lie(N)$ to a nilpotent subalgebra of $\End(V_\om)$.
Hence $\om(N)=\om(\exp(\Cn))=\exp(\om(\Cn))$ is unipotent.
\end{proof}

\begin{definition}
The representation $\om$ is called \e{unitary in the cusps}, if $\om(\ga)$ is unitary for every $\ga\in\Ga_N$ and every cuspidal $P=MAN$.
In that case, $\om$ is tame.
Even more is true, by the Lie-Kolchin Theorem, the representation $\om|_N$ diagonalises, i.e., it factors through the abelian quotient $N^\mathrm{ab}$.
\end{definition}

\section{Cusp forms}

\begin{definition}
For a minimal parabolic $\CP=MAN$ and  $T>0$ let $A_T$ denote the set of all $a\in A$ with $a^\al\ge T$ for every $\CP$-positive root of $A$.
\end{definition}

\begin{definition}\label{def2.1}
Let $\CP_1,\dots,\CP_r$ be the cusps of $\Ga$ and for each $1\le j\le r$ let $\CP_j=M_jA_jN_j$ be a Langlands decomposition.
We  write $\Ga_j$ for $\Ga_{N_j}=\Ga\cap N_j$.

Reduction theory \cites{Borel,Borel-Ji} tells us that there is $T>0$ and a compact set $\Om_T\subset G$, such that the set
\begin{align*}
\CS=\Om_T\cup\bigcup_{j=1}^r\underbrace{\CR_{\ j}\ A_{j,T}\ K}_{\CS_j}, 
\tag*{$(*)$}
\end{align*}
Contains the fundamental domain $\CF$, where $\CR_{\ j}$ is a relatively compact set of representatives of $\Ga\cap N_j\bs N_j$.
Note that \cites{Borel,Borel-Ji} requires the group $\Ga$ to be arithmetic.
If the splitrank of $G$ is $\ge 2$ and $\Ga$ irreducible, this condition is automatic \cite{Margulis}. If the splitrank of $G$ is one, the assertion is easy to show.
\end{definition}

\begin{definition}
For a cusp $\CP_j$ we write $\Ga_j$ for $\Ga\cap N_j$.
Let $\Pr_j:V_\om\to V_\om$ be the orthogonal projection onto the space $V_\om^{\Ga_j}$ of $\Ga_j$-invariants in the space $V_\om$.
\end{definition}

\begin{definition}
By super rigidity for unipotent groups, (Cor. 6.8 of \cite{Witte}), 
the representation $\om$ extends  from $\Ga_j'$ to $N_j$ for some finite-index subgroup $\Ga_j'$.
We fix such an extension.
A measurable function $\phi:G\to V_\om$ with $\phi(\ga x)=\om(\ga)\phi(x)$ for all $\ga\in\Ga$, $x\in G$, is called a \e{strong cusp form}, or simply a \e{cusp form}, if 
$$
\int_{\Ga_j'\bs N_j}\om(n^{-1})\(\phi(ny)\)\ dn=0
$$  
holds for all $y\in G$ and  $j=1,\dots, r$.
We write $L^2_\mathrm{cusp}(\Ga\bs G,\om)$ for the space of strong cusp forms in $L^2(\Ga\bs G,\om)$.
\end{definition}

\begin{definition}
Let $\Pr_j:V_\om\to V_\om$ denote the orthogonal projection onto the subspace $V_\om^{\Ga_j}$ of $\Ga_j$ invariants.
When dealing with a single cusp $\CP=MAN$, we write $\Pr_N$ instead of $\Pr_j$.
A measurable function $\phi:G\to V_\om$ with $\phi(\ga x)=\om(\ga)\phi(x)$ for all $\ga\in\Ga$, $x\in G$, is called a \e{weak cusp form}, if 
$$
\int_{\Ga_j\bs N_j}\Pr_j\(\phi(ny)\)\ dn=0
$$  
holds for all $y\in G$ and  $j=1,\dots, r$.
We write $L^2_\weak(\Ga\bs G,\om)$ for the space of weak cusp forms in $L^2(\Ga\bs G,\om)$.
\end{definition}

Note that, if $\om$ is unitary in the cusps, then every strong cusp form is also a weak cusp form.

\begin{remark}\label{remark2.6}
In the case that $\om$ extends to $G$, we relate the definition of strong cusp forms to the cusp forms in $L^2(\Ga\bs G, V_\om)$ via the map $\eta$ of Remark \ref{rem1.9}.
Now $\phi$ is a cusp form if for every $x\in G$ one has
$$
\int_{\Ga_j\bs N_j}\om(n^{-1})\(\phi(nx)\)\ dn=0
$$ 
and $f=\eta(\phi)$ satisfies the cusp condition $\int_{\Ga_j\bs N_j}f(nx)\,dn=0$ if
$$
\om(x^{-1})\int_{\Ga_j\bs N_j}\om(n^{-1})\(\phi(nx)\)\ dn=0
$$
and these two conditions coincide.
So $\eta$ preserves the cusp condition, although not the $L^2$-condition.
\end{remark}

\begin{remark}
In \cite{Venkov} and \cite{DeiMonEis}, what we call weak cusp form is called a cusp form. In both cases it is imposed that $\om$ be unitary in the cusps. If one wants to include the case of tame $\om$, then the notion of strong cusp forms is better suited.
This in particular justified by the nice comparison to untwisted cusp forms above.
\end{remark}

\begin{definition}
A function $\phi:\CF\to V_\om$ is called \e{rapidly decreasing}, if
for every $T>0$ one has
$$
\norm{\phi(x)}\ll \[\exp\[\norm{\log\ul a_j(x)}\]\]^{-T},
$$
where $\al$ runs through the $\CP_j$-positive roots on $A_j$ and $\ul a_j(x)$ is the $A_j$-part of $x$, when $x\in \CR_{\ j}\ A_j\ K$.

Note that if $\phi\in\CM(\Ga\bs G,\om)$, then the notion of $\phi$ being rapidly decreasing is independent of the choice of the fundamental domain $\CF$.
\end{definition}

\begin{theorem}\label{prop2.8}
Assume that $\om$ is tame.
Let $f\in C_c^\infty(G)$.
\begin{enumerate}[\rm(a)]
\item There exists $C>0$ such that for every $\phi\in L^2_\strong(\Ga\bs G,\om)$ we have
$$
\sup_{x\in\CF}\norm{R(f)\phi(x)}\le C\ \norm \phi_2.
$$
The function $R(f)\phi$ is rapidly decreasing on $\CF$.

\item The operator $R(f)$ is compact on $L^2_\cusp(\Ga\bs G,\om)$.
\item
The representation $R$ on $H=L^2_\cusp(\Ga\bs G,\om)$ is a direct sum of the subrepresentations $H_\la$, $\la\in\C^R$, i.e., the space
$
\ds\bigoplus_{\la\in\C^R}H_\la
$
is dense in $H$.
Only countably many of the $H_\la$ are non-zero and each $H_\la$ has a finite filtration with irreducible quotients.
\end{enumerate}
If $\om$ is unitary in the cusps, then all of these conclusions also hold for $L^2_\weak(\Ga\bs G,\om)$ instead of $L^2_\cusp(\Ga\bs G,\om)$.
\end{theorem}

\begin{corollary}
If $\om$ is unitary in the cusps and $V_\om^{\Ga_j}=0$ for each $j$, then $L^2(\Ga\bs G,\om)$ has no continuous spectrum.
This case can easily be established for $G=\SL_2(\R)$, since torsion-free, non-cocompact lattices are free groups in this case.
\end{corollary}

\begin{proof}
We shall give the proof for the case of strong cusp forms and remark the necessary changes or the weak case at the respective places.

(a)
If suffices to assume that $\phi$ is supported in $\CS_j$ for one fixed $j$. For brevity we write $N=N_j$ and $\Pr$ for $\Pr_j$.
For $\phi\in L^2_\cusp(\Ga\bs G,\om)$ we define the kernel $K(x,y)$ in the following computation:
\begin{align*}
R(f)\phi(x)&=\int_Gf(y)\phi(xy)\ dy=\int_Gf(x^{-1}y)\phi(y)\ dy\\
&=\int_{\Ga_{N}\bs G}\underbrace{\sum_{\ga\in\Ga_{N}}f(x^{-1}\ga y)\om(\ga)}_{K(x,y)}\ \phi(y)\ dy.
\end{align*}

Set $H(x,y)=\int_{N}f(x^{-1}ny)\ \om(n)\ dn$, where we choose the Haar measure on $N$ such that $\vol(\Ga_N\bs N)=1$.
Then
\begin{align*}
\int_{\Ga_N\bs G}H(x,y)\ \phi(y)\ dy
&=\int_{\Ga_N\bs G}\int_{N}f(x^{-1}ny)\om(n)\ dn\ \phi(y)\ dy\\
&=\int_{\Ga_N\bs G}\int_{N/\Ga_N}\sum_{\ga\in\Ga_N} f(x^{-1} n\ga y)\om(n)\ dn\ \om(\ga)\phi(y)\ dy\\
&=\int_{N/\Ga_N}\int_{\Ga_N\bs G}\sum_{\ga\in\Ga_N} f(x^{-1} n\ga y)\om(n)\ \phi(\ga y)\ dy\ dn\\
&=\int_{N/\Ga_N}\int_{G} f(x^{-1}n y)\om(n)\ \phi(y)\ dy\ dn\\
&=\int_{N/\Ga_N}\int_{G} f(x^{-1} y)\om(n)\ \phi(n^{-1}y)\ dy\ dn\\
&=\int_{G} f(x^{-1} y)\underbrace{\int_{N/\Ga_N}\om(n^{-1})\ \phi(ny)\ dn}_{0}\ dy.\\
\end{align*}
The case of weak cusp forms, one replaces $\om(n)$ by the projection $\Pr$ in the definition of $H$, ending up with the same result.

Set $K'=K-H$, then the corresponding integral operators of $K$ and $K'$ agree on $L^2_\cusp$.
By the theory of Malcev bases, there exists a linear bijection  $\La:\R^d\to \Cn=\Lie(N)$, such that with $\ul n(t)=\exp(\La(t))$ we have $\ul n(\Z^d)=\Ga_N$ and the image of the Lebesgue measure on $\R^d$ is a Haar measure on $N$.
Note that we now have a second group structures on $N$ with the same Haar measure. 
We denote this abelian Lie group as $N'$.

Note that if $\om|_N$ is unitary, it factors through $N^\mathrm{ab}$ by the Lie-Kolchin Theorem and hence $\om|_N$ in this case can be viewed as a representation of $N'$ as well.
For $x,y\in G$ let $F_{x,y}(t)=f\(x^{-1}\exp(\ul n(t))y\)\om(\ul n(t))$. 
Then $F_{x,y}\in C_c^\infty(\R^d,V_\om)$ and  $H(x,y)=\what F_{x,y}(0)$,
where $\what F_{x,y}$ is the Fourier transform of $F_{x,y}$.
The Poisson summation formula  implies
\begin{align*}
K'(x,y) &= \sum_{k\in\Z^d\sm\{0\}} \what F_{x,y}(k)\\ 
&=\  \sum_{\nu\in\Z^d\sm\{0\}}\int_{\R^d} f\(x^{-1}\ul n(t)\ y\)\ \om(\ul n(t))\ e^{-2\pi i\sp{\nu,t}}\,dt.
\end{align*}
As $f$ has compact support, $K(x,y)=0=K'(x,y)$  if the points $xK$ and $yK$ have large distance in $\CF$. Also, for given $x,y$ the integrand above has compact support.

By decreasing $T$ we can assume that if $K'(x,y)\ne 0$m then there is $j$ such that $x,y$ both lie in $\CF_j=\Om_T\cup\CS_j$.
We can write 
$x=r_xa_xk_x$ and likewise for $y$. There is a compact set $\kappa\subset A$ such that $a_y\in \kappa a_x$ and a compact set $C\subset\R^d$, such that
$a_x^{-1}\ul n(t)a_x\in \ul n\(x+C\)$.
We have
\begin{align*}
x^{-1}\ul n(t)y
&=k_x^{-1}n_x^{-1}a_x^{-1}\ul n(t)a_yn_yk_y\\
&=k_x^{-1}\ul n\(\Ad(n_x^{-1}a_x^{-1})t\)n_x^{-1}a_x^{-1}a_yn_yk_y\\
&=k_x^{-1}\ \ul n(\al_x^{-1}(t))\ \be_{x,y},
\end{align*}
where $k_x\in K$ and $\be_{x,y}$ stays in a fixed compact set. The map $\al_x=\Ad(a_xn_x)$ is a linear bijection on the Lie algebra, depending  smoothly on $x$.
Note that $n_x$ stays in a fixed compact set and therefore the norm of any $\al_x(v)$ is ruled by the eigenvalues of $\Ad(a_x)$ which are bounded below for $x\in\CF_j$.

\begin{lemma}\label{lemm2.9}
As $x$ tends to infinity in $\CF_j$, we have
$$
\exp\[a\norm{\log\ul a_j(x)}\]\ll  \norm{\al_x}\ll\exp\[b\norm{\log\ul a_j(x)}\],
$$
for some $0<a<b$,
where $\ul a_j(x)$ is the $A$-part in the Iwasawa decomposition $x=a_jn_jk_j$ and $\al$ ranges over the $\CP_j$-positive roots.
\end{lemma}

\begin{proof}
Use the root space decomposition of $\Cn=\Lie(N)$.
\end{proof}

Changing variables to $\al_x^{-1}(t)$ we get
\begin{align*}
K'(x,y)&=\sum_{\nu\ne 0}\int_{\R^d} |\det(\al_x)|\ f\(k_x\ \ul n(t)\ \be_{x,y}\)\ \om\(\ul n(\al_x(t)\)\ e^{-2\pi i\sp{\nu,\al_x(t)}}\ dt\\
&=|\det(\al_x)|\sum_{\nu\ne 0}\int_{\R^d} \ f\(k_x\ \ul n(t)\ \be_{x,y}\)\ \om\(\ul n(\al_x(t))\)\ e^{-2\pi i\sp{\al_x(\nu)^*,t}}\ dt.
\end{align*}
The function 
$$
X\mapsto\int_\R f\(k_x\ \ul n(t)\ \be_{x,y}\)\ \om\(\ul n(\al_x(t))\)\ e^{-2\pi i\sp{X,t}}\ dt
$$ 
is the Fourier transform of a Schwartz function, hence a Schwartz function itself.
Since $\om$ is tame, the operator norm $\norm{\om\(\ul n(\al_x(t))\)}$ is bounded by a polynomial in $\norm t$ and $\norm{\al_x}$.
Hence it follows that for every $n\in\N$ there exists $C_n>0$ such that 
\begin{align*}
\norm{K'(x,y)}\ &\le\ C_n\ |\det(\al_x)|\norm{\al_x}^A\sum_{\nu\ne 0}\(1+\norm{\al_x(\nu)}\)^{-n}\\
&\le C_n'\ \norm{\al_x}^{A+1}\sum_{\nu\ne 0}\(1+\exp\[a\norm{\log\ul a_j(x)}\]\norm \nu\)^{-n}.
\end{align*}
This is the decisive estimate. For the case of  weak cusp forms, in the Fourier expansion of $K'$ one has to add a summand for $\nu=0$ which corresponds to the projection $(1-\Pr)$.
As in this case, $\om|_N$ can be viewed as a representation of $N'$,
the summand of $\nu=0$ is a sum of non-trivial characters in $t$, so the Schwartz function argument works as well.

Hence there is $E>0$ and for every $n$ there is a constant $D(n)>0$ such that for $\phi\in L^2_\cusp(\Ga\bs G,\om)$ we have
\begin{align*}
\norm{K'(x,y)}\le D(n)\(1+E\ \exp\[a\norm{\log\ul a_j(x)}\]^{-n}
\end{align*}
The Cauchy-Schwartz inequality implies
\begin{align*}
\norm{R(f)\phi(x)}
&\le D(n)\(1+E\ \exp\[a\norm{\log\ul a_j(x)}\]^{-n}\norm \phi_2
\end{align*}
This implies part (a).
To prove part (b), we  use the Arzela--Ascoli Theorem.
Let $f\in C_c^\infty(G)$.
We have
$$
R(f)L^2( \Ga\bs G,\om)\ \subset\ C^\infty(\Ga\bs G,\om).
$$
By part (a) the image of $\left\{\phi\in L_\cusp^2:\norm\phi_2=1\right\}$ is globally bounded and likewise for the image of 
$\tilde R_XR(f)=R(\tilde L_Xf)$ for every $X\in \mathrm{Lie}(G)$, where $\tilde R$ denotes the derivative along $X$.
By the Arzela--Ascoli Theorem every sequence
in $R(f)L_\cusp^2$ has a point-wise convergent subsequence and by part (a) this sequence is dominated by a constant.
Therefore the sequence converges in the 
$L^1$-norm and as it is bounded, it converges in the $L^2$-norm.
So the operator
$R(f)$ on $L_\cusp^2$ is indeed compact.
This proves part (b) of the theorem.
Finally, part (c) follows from the Theorem \ref{nextprop}.
\end{proof}

\section{Incomplete Eisenstein   series}

\begin{definition}
We fix a torsion-free lattice $\Ga$ of $G$ and a maximal compact subgroup $K$ of $G$.
If $T>0$ is large enough, then the space $\Ga\bs G$ has the following form
$$
\Ga\bs G=\Om\cup\bigcup_{j=1}^n S_j,
$$
where $\Om$ is compact and the $S_j$ are the \e{cusp sections}, each of the form $S_j\cong \Ga_{P_j}\bs N_jA_{j,T}K$, where $P_j=M_jA_jN_j$ are the cusps.
Enlarging $\Om$ if necessary, we can assume the cusp sections to be pairwise disjoint.
\end{definition}

\begin{definition}
Fix a cusp $ P_j= P=MAN$.
Let $W_P$ denote the space of all smooth functions $f:G\to\C$, for which there exists $\al\in C_c^\infty(A)$, such that 
$$
f(anx)=\al(a)\ f(x)
$$
holds for all $x\in G$ and all $an\in AN$.
\end{definition}

\begin{definition}
For a cusp $ P_j= P=MAN$ a function $f\in W_P$ and a vector $v\in V_\om^{\Ga_N}$ 
we define the   \e{incomplete Eisenstein  series} as
$$
E_{f,v}(x)=\sum_{\ga\in\Ga_N\bs \Ga}f(\ga x)\ \om(\ga^{-1})v.
$$
Note that, since $\al$ has compact support, the sum is locally finite.
Let 
$
C_{c,\Eis}^\infty(\om)
$
denote the vector space of all incomplete Eisenstein  series.
\end{definition}

\begin{theorem}
Let $L^2_{\Eis}(\om)$ denote the $L^2$-closure of the space of incomplete Eisenstein  series. 
We have  $G$-stable decompositions
$$
C_c^\infty(\om)\ =\ C_{c,\Eis}^\infty(\om)\ \oplus\ C_{c,\cusp}^\infty(\om)
$$
and
$$
L^2(\om)\ =\ L^2_{\Eis}(\om)\ \oplus\ L^2_\cusp(\om).
$$
\end{theorem}

\begin{proof}
As $C_{c,\Eis}^\infty(\om)$ and $C_{c,\cusp}^\infty(\om)$ are dense in the corresponding $L^2$ spaces, it suffices to prove the following:
\begin{enumerate}[\rm(a)]
\item There is an open covering $\Ga\bs G=\bigcup_{j=1}^\infty U_j$ such that for every $j$ there is a choice of a metric $\sp{.,.}_x$ in the given metric class, such that $C_{c,\Eis}^\infty(\om)$ is orthogonal to the space of all $\phi\in L^2_\cusp(\om)$ with support in $U_j$.
\item 
$$
C_c^\infty(\om)=C_{c,\Eis}^\infty(\om)+ C_{c,\cusp}^\infty(\om).
$$
\end{enumerate}
To show (a), we reduce it further.
Suppose there exists an open set $U$ such that $C_{c,\Eis}^\infty(\om)$ is orthogonal to all $\phi\in L_\cusp^2(\om)$ with support in $U$.
Then, as $C_{c,\Eis}^\infty(\om)$ and $L^2_\cusp(\om)$ both are $G$-stable, the same holds for $Ux$ for any given $x\in G$.
The space $\Ga\bs G$ can be covered with countable many such translates.

For $\phi\in L^2(\om)$ we compute
\begin{align*}
\sp{\phi,E_{f,v}}
&=\int_{\Ga\bs G}\sp{\phi(x),E_{f,v}(x)}_x\ dx\\
&=\int_{\Ga\bs G}\sum_{\ga\in\Ga_n\bs \Ga}\ol f(\ga x)\sp{\phi(x),\om(\ga^{-1})v}_x\ dx\\
&=\int_{\Ga\bs G}\sum_{\ga\in\Ga_n\bs \Ga}\ol f(\ga x)\sp{\phi(\ga x),v}_{\ga x}\ dx\\
&=\int_{\Ga_N\bs G}\ol f(x)\sp{\phi(x),v}_x\ dx\\
&=\int_{\Ga_N\bs NAK}a^{2\rho}\ \ol f(nak)\sp{\phi(nak),v}_{nak}\ dn\ da\ dk\\
&=\int_{\Ga_N\bs NAK}a^{2\rho}\ \ol\al(a)\ol f(k)\sp{\phi(nak),v}_{na}\ dn\ da\ dk.
\end{align*}
Choose $t>0$ large enough, such that the set $\Ga_N\bs NA_tK$ makes a cusp section.
Then, as $\om $ is unitary in the cusps, for $a\in A_t$ we may alter the inner product so as to satisfy $\sp{.,.}_{na}=\sp{.,.}_a$.
Suppose that $\phi$ is supported in $NA_tK$ modulo $\Ga$, then we get
\begin{align*}
\sp{\phi,E_{f,v}}
&=\int_{\Ga_N\bs NAK}a^{2\rho}\ \ol\al(a)\ol f(k)\sp{\phi(nak),v}_{a}\ dx\\
&=\int_{AK}a^{2\rho}\ \ol\al(a)\ol f(k)\sp{\int_{\Ga_N\bs N}\Pr \phi(nak)\ dn,v}_{a}\ da\ dk.
\end{align*}
If $\phi$ is a cusp form, the latter is zero.
This means that the above condition is satisfied for $U$ being a cusp section.
Hence incomplete Eisenstein  series are orthogonal to cusp forms.

For (b) we use essentially the same argument.
Let $\phi\in C_c^\infty(\om)$ and define a map $A_\phi:G\to V_\om^{\Ga_N}$ by 
$$
A_\phi(x)=\int_{\Ga_N\bs N}\Pr(\phi(nx))\ dn.
$$
Then the map $\phi\mapsto A_\phi$ is $G$-equivariant.
If $\phi$ is supported in the cusp section $S=\Ga_N\bs NA_tK$ as above, then we can find an incomplete Eisenstein  series $E=E_{f,v}$ such that $A_\phi=A_E$ in that cusp section.
We can further choose the support of the function $f$ small enough so that the support of $E_{f,v}$ does not exceed the support of $f$.
This means that $\phi$ equals an incomplete Eisenstein  series plus a cusp form.
For arbitrary $\phi$, consider $R_x\phi$ for $x\in G$. One can choose $x\in G$ such that $R_x\phi$ has support in the cusp section and thus the proof is finished.
\end{proof}

\begin{definition}
Let $\Cz$ denote the center of the universal enveloping algebra $U(\Cg)$.
Let $\CA(\Ga,\om)$ denote the space of all $\om$-automorphic, smooth, rapidly decreasing, $K$-finite and $\Cz$-finite functions $G\to V_\om$.
Likewise, let $\CA(\Ga\bs G)$ denote the space of functions $f:G\to\C$ satisfying the same properties.

In both cases, let $\CA_\cusp$ denote the respective subspace of cusp forms.
\end{definition}

\begin{proposition}
Assume that $\om$ extends to a $G$-representation.
The map $\eta$, defined by $\eta\(\phi(x)\)=\om(x^{-1})\phi(x)$ restricts to bijections
\begin{align*}
\eta:\CA(\Ga,\om)&\tto\cong\CA(\Ga\bs G)\otimes V_\om,\\
\eta:\CA_\cusp(\Ga,\om)&\tto\cong\CA_\cusp(\Ga\bs G)\otimes V_\om.
\end{align*}
The space $\CA_\cusp(\Ga,\om)$ is dense in $L^2_\cusp(\Ga,\om)$ and $\CA_\cusp(\Ga\bs G)$ is dense in $L^2_\cusp(\Ga\bs G)$.
\end{proposition}

\begin{proof}
Clear.
\end{proof}

\begin{remark}
Incomplete Eisenstein series give a $G$-Homomorphism from the induced representation
$$
\mathrm{Ind}_{P}^G\(C_c^\infty(A)\)\ \otimes\ V_\om^{\Ga_N}
$$
to $L^2(\om)$.
Hence the complement of the cusp forms can be identified with this induced representation.

Now $C_c^\infty(A)$ is a subset of the Hilbert space $L^2(A)$, which, as a representation space of $A$, is a direct Hilbert integral $\int_{\what A}^\oplus\C_\la\, d\la$ of one-dimensional spaces.
If this construction were to commute with induction and with forming Eisenstein series, one would get an injection
$$
\int_{\what A}^\oplus \C_\la\ d\la\hookrightarrow L^2(\om),
$$
describing the complement of the cusp forms as a direct Hilbert integral.
However, there are normalisation and convergence issues, which through a thorough analysis of Eisenstein series can be overcome in the classical situation $\om=1$. 
Even in this situation one needs analytic continuation of Eisenstein series, and one catches some additional discrete representations at the poles of the Eisenstein series.
In the current situation $\om\ne 1$, these techniques are not available.

There are reasons to believe that they never will.
For instance, in the case $G=\SL_2(\R)$, the structure of $\Ga$ is as follows: There are hyperbolic generators $\ga_1,\dots,\ga_{2g}$ and parabolic generators $p_1,\dots,p_r$ and we have only one relation
$$
[\ga_1,\ga_2]\cdots[\ga_{2g-1},\ga_{2g}]p_1\cdots p_r=1.
$$
The number $r$ is the number of cusps.
We assume $r\ge 1$, so $\Ga$ is a free group and we can prescribe the values $\om(\ga_j)$ and $\om(p_j)$ arbitrarily.
Now consider the case, when $\om$ extends to $G$.
We get the continuous spectrum in $L^2(\om)=L^2\otimes\om$ by taking the continuous spectrum of the case $\om=1$ and translating the verticals which parametrise the continuous spectrum, by the weights of $\om$.

Deforming the representation $\om$ continuously to the trivial representation, what happens to the continuous spectrum?
The easiest would be, that the vertical lines stay vertical, but move to the imaginary axis, where they finally join.
But why should they stay vertical? Any other contour may occur or even the continuous spectrum might dissolve completely.
As long as there is no idea around, what these phenomena should entail to for non-unitary $\om$, there is little hope to describe the complement of the cusp form space in representation-theoretic terms.

\end{remark}

{\small Mathematisches Institut\\
auf der Morgenstelle 10\\
72076 T\"ubingen\\
Germany\\
\tt deitmar@uni-tuebingen.de}

\today

\end{document}